\documentclass{amsart}
\usepackage[cp1251]{inputenc}
\usepackage[russian,english]{babel}
\usepackage{amsmath}
\usepackage{amssymb}
\usepackage{amsfonts}

\usepackage{xcolor}
\usepackage[all]{xy}

\usepackage{graphicx,graphics,hhline}
\RequirePackage{caption2}

\addto\captionsrussian{}

\newtheorem{lemma}{Lemma}
\newtheorem{theorem}{Theorem}
\newtheorem{definition}{Definition}
\newtheorem{corollary}{Corollary}
\newtheorem{proposition}{Proposition}

\begin{document}

\thispagestyle{empty}

\title[On weakly $1$-convex and weakly $1$-semiconvex sets]{On weakly $1$-convex and weakly $1$-semiconvex sets}


\author{ Tetiana M.  Osipchuk}
\address{Tetiana M. Osipchuk 
\newline\hphantom{iii} Institute of Mathematics of the National Academy of Sciences of Ukraine,
\newline\hphantom{iii} Tereshchenkivska str. 3,  
\newline\hphantom{iii} UA-01004, Kyiv, Ukraine}

\thanks{\sc
On weakly $1$-convex and weakly $1$-semiconvex sets}
\thanks{\copyright \ 2024 T.M. Osipchuk}


\maketitle
{
\small
\begin{quote}
\noindent{\bf Abstract. } The present work concerns generalized convex sets in the real multi-dimensional Euclidean space, known as weakly $1$-convex and weakly $1$-semiconvex sets. An open set is called \emph{weakly $1$-convex} (\emph{weakly $1$-semiconvex}) if, through every boundary point of the set, there passes a straight line (a closed ray) not intersecting the set. A closed set is called \emph{weakly $1$-convex} (\emph{weakly $1$-semiconvex}) if it is approximated from the outside by a family of open weakly $1$-convex (weakly $1$-semiconvex) sets. A point of the complement of a set to the whole space is a \emph{$1$-nonconvexity} (\emph{$1$-nonsemiconvexity}) \emph{point} of the set  if every straight line passing through the point (every ray emanating from the point) intersects the set. It is proved that if the collection of all $1$-nonconvexity ($1$-nonsemiconvexity) points corresponding to an open weakly $1$-convex (weakly $1$-semiconvex) set is non-empty,  then  it is open. It is also proved that the non-empty interior of a closed weakly $1$-convex (weakly $1$-semiconvex) set in the space is weakly $1$-convex (weakly $1$-semiconvex).

 \medskip

\end{quote}
}

\section{Introduction}

The weakly $m$-convex and weakly $m$-semiconvex sets, $m=1,2,\ldots,n-1$, in the real space $\mathbb{R}^n$, $n\geqslant2$, with the Euclidean norm, can be seen as a generalization of convex sets. The notions were coined by Yurii Zelinskii \cite{Zel3}, \cite{Zel01}. First, recall the following definitions.

Any $m$-dimensional affine subspace of  $\mathbb{R}^n$, $m=0,1,2,\ldots,n-1$, $n\geqslant1$, is called an {\bf\emph{$m$-dimensional plane}}. A $1$-dimensional plane is also known  as a {\bf\emph{straight line}}.

One of two parts of an $m$-dimensional plane, $m=1,2,\ldots,n-1$, of the space $\mathbb{R}^n$, $n\geqslant2$, into which it is divided by any of its $(m-1)$-dimensional planes (herewith, the points of the $(m-1)$-dimensional plane are included) is said to be an {\bf\emph{ $m$-dimensional half-plane}}. A $1$-dimensional half-plane is also known  as a {\bf\emph{ray}}.

\begin{definition}[Zelinskii \cite{Zel3}, \cite{Zel01}]\label{def3}  An open subset  $E\subset\mathbb{R}^n$, $n\geqslant2$, is called \textbf{weakly $m$-convex} (\textbf{weakly $m$-semiconvex}),  $m=1,2,\ldots,n-1$, if for any point $x\in\partial E$, there exists
an $m$-dimensional plane $L$ ($m$-dimensional half-plane $L$) such that $x\in L$ and $L\cap E=\varnothing$.
\end{definition}

They say that a set $A$ {\bf\emph{is approximated from the outside}} by a family of open sets $A_k$,
$k=1,2,\ldots$, if $\overline{A}_{k+1}$ is contained in $A_k$, and $A=\cap_kA_k$ (\cite{Aiz3}).

It can be proved that any set approximated from the outside by a
family of open sets is closed.
\begin{definition}[Zelinskii \cite{Zel3}, \cite{Zel01}]\label{def4} A
closed subset $E\subset\mathbb{R}^n$, $n\geqslant2$, is called \textbf{weakly $m$-convex} (\textbf{weakly $m$-semiconvex}),  $m=1,2,\ldots,n-1$, if it can be approximated
from the outside by a family of open weakly $m$-convex (weakly $m$-semiconvex) sets.
\end{definition}

The class of weakly $m$-convex sets in $\mathbb{R}^n$ is denoted by $\mathbf{WC^n_m}$ and the class of weakly $m$-semiconvex sets in $\mathbb{R}^n$ is denoted by $\mathbf{WS^n_m}$.

The properties of the class of generalized convex sets on Grassmannian manifolds which are closely related to the properties of the conjugate sets (see \cite[Definition 2]{Zel01}) are investigated in \cite{Zel01}. This class includes $\mathbf{WC^n_m}$. The geometric and topological properties of weakly $m$-convex sets are also investigated in \cite{Dak}, \cite{Dak01}.

The theory of weakly $m$-semiconvex sets is newish and it is based on the research of some subclass as well as further investigation of weakly $m$-convex sets also focuses on the similar subclass. In order to determine these subclasses, we need to set the following definition.

\begin{definition}\label{def5}
A point $x\in \mathbb{R}^n\setminus E$ is called an {\bf\emph{$m$-nonconvexity}} ({\bf\emph{$m$-nonsemiconvexity}}) {\bf\emph{point}} of a subset $E\subset\mathbb{R}^n$ if every $m$-dimensional plane ($m$-dimensional half-plane) passing through $x$ intersects $E$. The set of all $m$-nonconvexity ($m$-nonsemiconvexity) points of a subset  $E\subset\mathbb{R}^n$ is called the {\bf\emph{$m$-nonconvexity-point}} ({\bf\emph{$m$-nonsemiconvexity-point}})  {\bf\emph{set}} corresponding to $E$ and is denoted by $E_m^{\triangle}$ ($E_m^{\diamondsuit}$). Moreover, $E^{\triangle}:=E_1^{\triangle}$, $E^{\diamondsuit}:=E_1^{\diamondsuit}$.
\end{definition}

The class of weakly $m$-convex sets in $\mathbb{R}^n$ with non-empty $m$-nonconvexity-point set is denoted by $\mathbf{WC^n_m}\setminus \mathbf{C^n_m}$ and the class of weakly $m$-semiconvex sets with non-empty $m$-nonsemiconvexity-point set in $\mathbb{R}^n$ is denoted by $\mathbf{WS^n_m}\setminus \mathbf{S^n_m}$.

The disconnectedness of any open weakly $1$-semiconvex set with non-empty $1$-nonsemiconvexity-point set in the plane was established by Zelinskii \cite[Theorem~7]{Zel3}. Moreover, the following result is true.

\begin{lemma}[Dakhil \cite{Dak}, Osipchuk \cite{Osi2022}]\label{theor01}
An open set or a closed set belonging to the class $\mathbf{WS^2_{1}}\setminus \mathbf{S^2_{1}}$ consists of not less than three connected components.
\end{lemma}

Interestingly, the number of components of a set belonging to the class $\mathbf{WS^2_{1}}\setminus \mathbf{S^2_{1}}$ is also affected by the smoothness of its boundary.

\begin{lemma}[Osipchuk \cite{Osipchuk2019}]\label{theor3}
Suppose that  an open bounded subset $E\subset \mathbb{R}^2$ with smooth boundary belongs to the class $\mathbf{WS^2_1}\setminus \mathbf{S^2_1}$. Then $E$
consists of not less than four connected components.
\end{lemma}
\begin{lemma}[Osipchuk \cite{Osi2022}]\label{theor3_1}
Suppose that  a closed bounded subset $E\subset \mathbb{R}^2$ with smooth boundary and such that $\mathrm{Int}\, E$ is not $1$-semiconvex belongs to the class $\mathbf{WS^2_1}\setminus \mathbf{S^2_1}$. Then $E$
consists of not less than four connected components.
\end{lemma}
  \begin{figure}[h]
	\centering
   \includegraphics[width=11 cm]{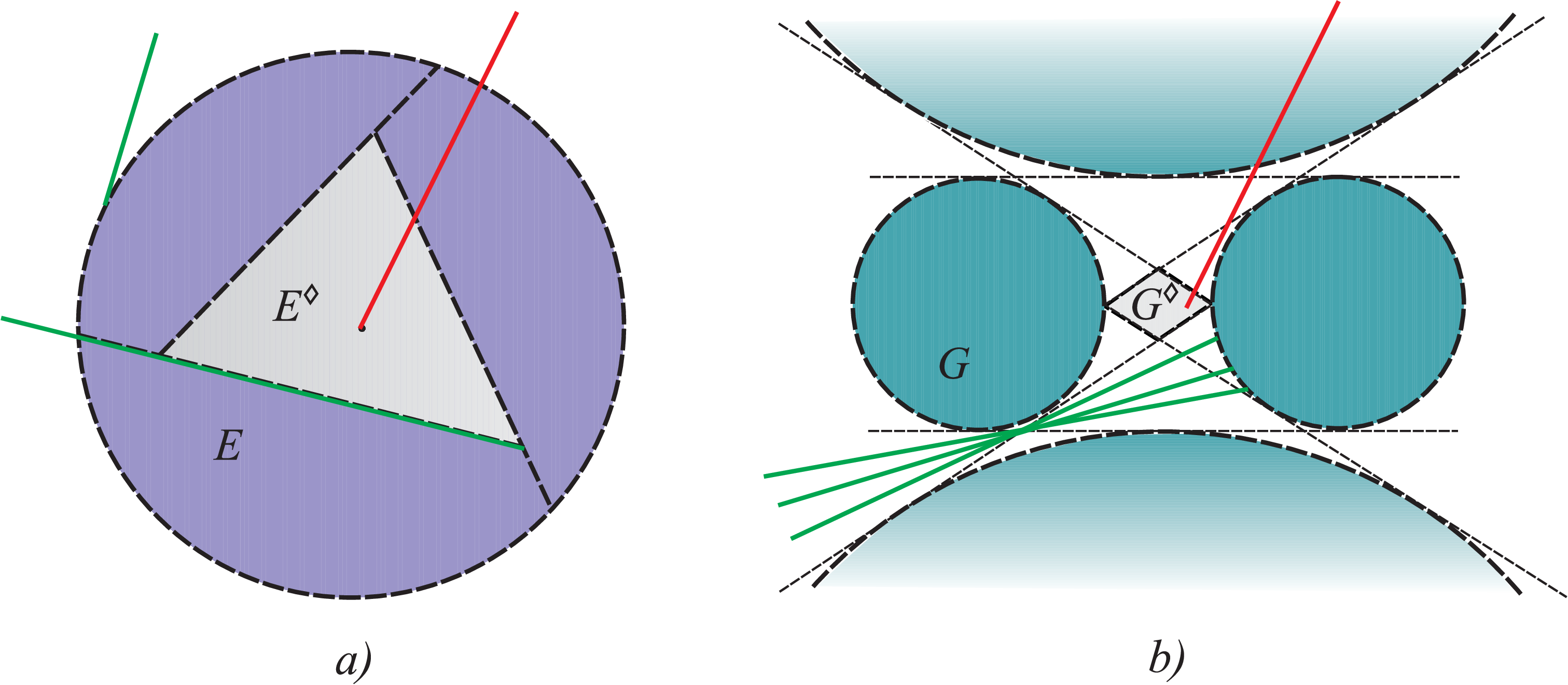}
	\caption{}\label{Fig17}
\end{figure}

The example of an open set $E\in\mathbf{WS^2_{1}}\setminus \mathbf{S^2_{1}}$ consisting of three components is in Figure~\ref{Fig17} a), and an open set $G\in\mathbf{WS^2_{1}}\setminus \mathbf{S^2_{1}}$ with smooth boundary and four components is in Figure~\ref{Fig17} b).  Moreover, if we want to construct  an open set belonging to the class $\mathbf{WS^2_{1}}\setminus \mathbf{S^2_{1}}$ with countably infinite number of components, then, instead of a triangle inside a convex set, we should  throw away a closed convex generalized polygon (the convex hull of a bounded countably infinite set of points in the plane with boundary containing countably infinite number of vertices). The example of a closed convex generalized polygon is the convex hull of the points
$$
y_0,\, y_\pi,\, y_{\frac{\pi}{2}},\, y_{2\pi-\frac{\pi}{2}},\, y_{\frac{\pi}{4}},\, y_{2\pi-\frac{\pi}{4}},\,\ldots,\, y_{\frac{\pi}{2^k}},\, y_{2\pi-\frac{\pi}{2^k}},\,\ldots
$$
in Figure~\ref{Fig18} a).  And also cut the obtained set along rays containing the polygon sides and the accumulation points of the polygon vertices as it is shown in Figure~\ref{Fig18} b).

Examples of closed sets belonging to $\mathbf{WS^2_{1}}\setminus \mathbf{S^2_{1}}$ with non-smooth or smooth boundary see in \cite{Osi2022}.

\begin{figure}[h]
	\centering
   \includegraphics[width=11 cm]{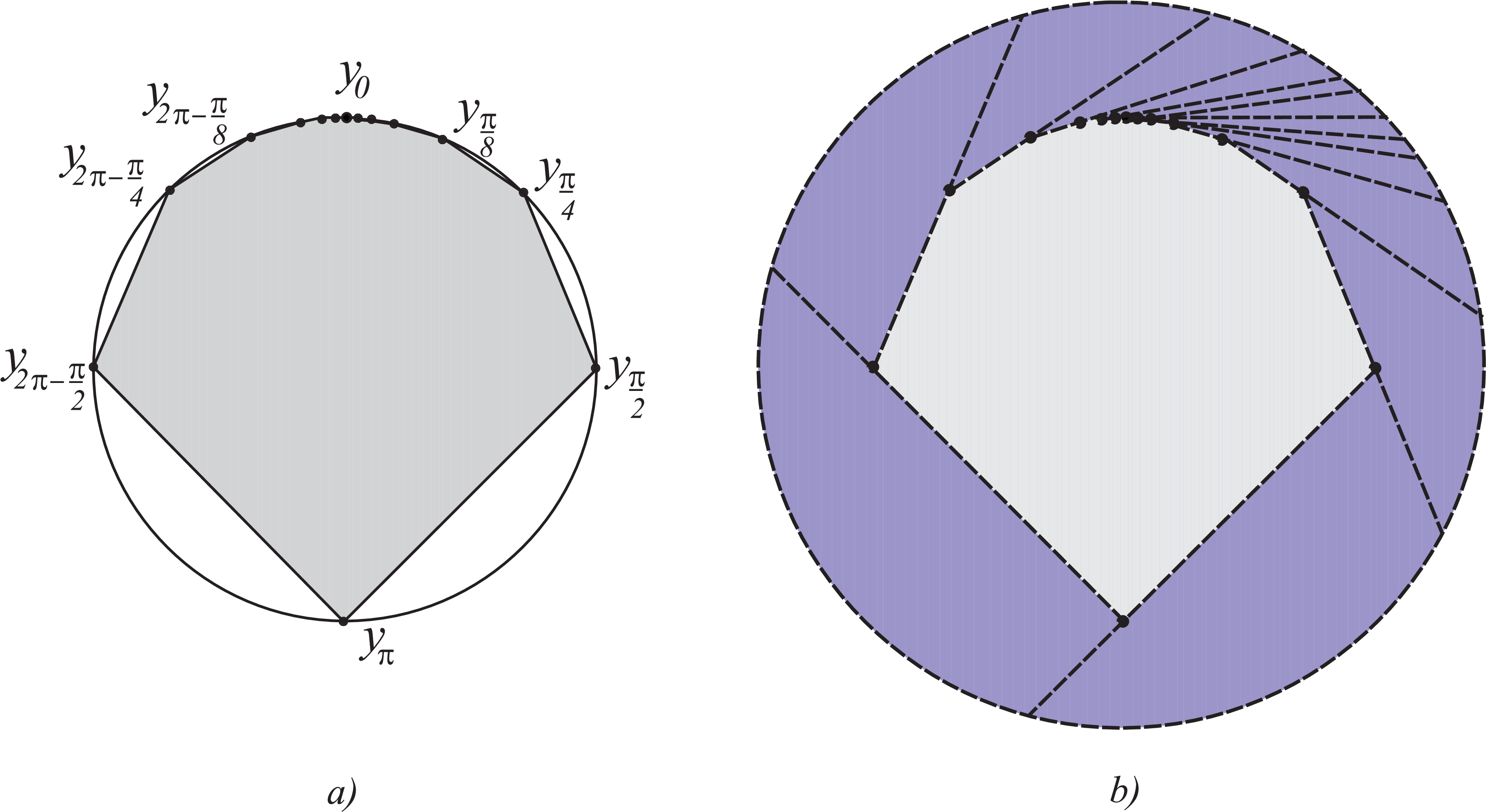}
	\caption{}\label{Fig18}
\end{figure}

Notice that the above properties of weakly $m$-semiconvex sets could so far be established only in the plane, in contrast to weakly $m$-convex sets.

\begin{lemma}[Dakhil \cite{Dak}, Osipchuk \cite{Osi2022_3}]\label{theor01}
An open set or a closed set that belongs to the class $\mathbf{WC^n_{n-1}}\setminus \mathbf{C^n_{n-1}}$ consists of not less than three connected components.
\end{lemma}

But unlike weakly $1$-semiconvex sets with smooth boundary, {\it any open weakly $(n-1)$-convex set in $\mathbb{R}^n$ with smooth boundary does not have $(n-1)$-nonconvexity points}  \cite[Proposition~2.3.7]{Dak}.

An example of sets belonging to the class $\mathbf{WC^2_{1}}\setminus \mathbf{C^2_{1}}$ can be constructed by cutting an open convex set without closed convex polygon or generalized polygon in Figures~\ref{Fig17} a) and ~\ref{Fig18} b) along the straight lines containing the sides and the accumulation points of vertices of the polygons instead of rays. Examples of open and closed sets belonging to $\mathbf{WC^n_{n-1}}\setminus \mathbf{C^n_{n-1}}$ see in \cite{Osi}.

For $n\geqslant 3$ and $m=1,2,\ldots,n-2$, the disconnectedness property is violated both for weakly $m$-convex and for weakly $m$-semiconvex sets.

 \begin{lemma}[Osipchuk \cite{Osi2022_3, Osi2022}]\label{theor02}
There exist domains and closed connected sets in the space $\mathbb{R}^n$, $n\geqslant 3$, belonging to the class $\mathbf{WC^n_m}\setminus \mathbf{C^n_m}$ ($\mathbf{WS^n_m}\setminus \mathbf{S^n_m}$), $1\leqslant m<n-1$.
\end{lemma}

Of special interest are the properties of $m$-nonconvexity-point sets corresponding to weakly $m$-convex sets and $m$-nonsemiconvexity-point sets corresponding to weakly $m$-semiconvex sets. The following results were obtained.
\begin{lemma}[Osipchuk \cite{Osi2022_2, Osi2023}]\label{lemm3}
Suppose that an open subset $E\subset \mathbb{R}^2$ belongs to the class $\mathbf{WC^2_1}\setminus \mathbf{C^2_1}$. Let $E^{\triangle}_j$, $j\in N\subseteq\mathbb{N}$, be the components of $E^{\triangle}$. Then
 \begin{enumerate}
 \item[\textup{(a)}] $E^{\triangle}$ is open and weakly $1$-convex;
 \item[\textup{(b)}] $E^{\triangle}_j$, $j\in N$, are convex (bounded or unbounded);
 \end{enumerate}
 \end{lemma}

\begin{lemma}[Osipchuk \cite{Osi2024}]\label{lemm4}
Suppose that an open subset $E\subset \mathbb{R}^2$ belongs to the class $\mathbf{WS^2_1}\setminus \mathbf{S^2_1}$. Let $E^{\diamondsuit}_j$, $j\in N\subseteq\mathbb{N}$, be the components of $E^{\diamondsuit}$. Then
 \begin{enumerate}
 \item[\textup{(a)}] $E^{\diamondsuit}$ is open and weakly $1$-semiconvex;
 \item[\textup{(b)}] $E^{\diamondsuit}_j$, $j\in N$, are convex and bounded;
 \item[\textup{(c)}] any connected subset of $\partial E^{\diamondsuit}_j$, $j\in N$, consisting of only smooth points is a line segment or a point;
 \item[\textup{(d)}]  there exists a collection of rays $\left\{\eta^{k}\right\}_{k\in M}$, $M\subseteq\mathbb{N}$, such that
 \begin{enumerate}
     \item[$\bullet$] $\bigcup\limits_k \eta^{k}\supset \partial E^{\diamondsuit},$
 \item[$\bullet$] the set $\bigcup\limits_k \eta^{k}\bigcup E^{\diamondsuit}$ does not contain rays emanating from $E^{\diamondsuit}$,
\item[$\bullet$] $\bigcup\limits_k \eta^{k}\bigcap E=\varnothing.$
 \end{enumerate}
 \end{enumerate}
 \end{lemma}
 In other words, Lemma~\ref{lemm4} shows that the $1$-nonsemiconvexity-point set corresponding  to a flat weakly $1$-semiconvex set is the union of open convex polygons and open convex generalized polygons. But they cannot be arbitrarily placed in the plane. Their arrangement is constrained by property (d).

 The methods developed to prove item (a) in Lemmas \ref{lemm3} and \ref{lemm4} allow us to obtain the following result for the closed weakly $1$-convex (weakly $1$-semiconvex) sets in the plane.

 \begin{lemma}[Osipchuk \cite{Osi2024, Osi2023}]\label{lemm6}
Let $E\subset \mathbb{R}^2$ be a closed subset such that $\mathrm{Int}\, E\ne\varnothing$. If $E$ is weakly $1$-convex (weakly $1$-semiconvex), then $\mathrm{Int}\, E$ is weakly $1$-convex (weakly $1$-semiconvex).
\end{lemma}

In this study, we focus on establishing the general topological properties of the $1$-nonconvexity-point set corresponding to an open weakly $1$-convex set and the $1$-nonsemiconvexity-point set corresponding to an open weakly $1$-semiconvex set in $\mathbb{R}^n$, $n\geqslant 2$.

First, we prove that {\it The $1$-nonsemiconvexity-point set $E^{\diamondsuit}$ corresponding to an open  set $E\in \mathbf{WS^n_1}\setminus \mathbf{S^n_1}$, $n\geqslant 2$, is open}. Therefore, we generalize Lemma~\ref{lemm4} (a) on the real Euclidean space of any dimension $n\geqslant 2$. The proof algorithm is similar to the proof of Lemma~\ref{lemm4} (a). Its essence is to find, for every fixed point $y\in E^{\diamondsuit}$ and each ray emanating from $y$,  points $x_\alpha(y)$, $\alpha\in S^{n-1}$, on these rays, and the number $d(y)>0$ such that the points $x_\alpha(y)$ are contained in $E$ together with open balls of the same radii $d(y)$. This allows us to assert that any ray emanating from an open ball with center at $y$ and radius $\varepsilon\le d(y)$ intersects the union of the balls contained in $E$. Thus, we show that $y$ is an inner point of $E^{\diamondsuit}$.

To find $x_\alpha(y)\in E\subset \mathbb{R}^2$, $\alpha\in [0,2\pi]$, it was used  the connectedness of the components of $E$. Namely, there were constructed the finite number of curves contained in $E$ and such that every ray emanating from $y$ intersects the union of the curves. Moreover, it was shown that points $x_\alpha(y)$ are actually placed on those curves and $d(y)$ is the minimum value of the restrictions of the distance functions defined on the components of $E$ to the respective curves. This trick fails for the set $E$ in the spaces of higher dimensions, obviously. But we are lucky to find not one-dimensional compacts that meet our requirements.

Using the same algorithm, we also prove that {\it The $1$-nonconvexity-point set $G^{\triangle}$ corresponding to an open set $G\in \mathbf{WC^n_1}\setminus \mathbf{C^n_1}$, $n\geqslant2$, is open}. But in this case, we show that, for every fixed point $y\in G^{\triangle}$ and each straight line passing through $y$,  there exist points $x_\alpha(y)$, $\alpha\in S^{n-1}$, on these lines, and the number $d(y)>0$ such that the points $x_\alpha(y)$ are contained in $G$ together with open balls of the same radii $d(y)$.

The consequents of these two statements are that $E^{\diamondsuit}$ is weakly $1$-semiconvex  for $E\in\mathbf{WS^n_1}\setminus \mathbf{S^n_1}$ and $G^{\triangle}$ is weakly $1$-convex for $G\in\mathbf{WC^n_1}\setminus \mathbf{C^n_1}$, $n\geqslant2$.

The methods developed to prove the first two results allow us to generalize Lemma~\ref{lemm6} on closed weakly $1$-semiconvex and closed weakly $1$-convex sets in $\mathbb{R}^n$, $n\geqslant2$.

Property (d) of Lemma~\ref{lemm4} easily extends to all spaces with dimensions $n\geqslant2$, and is also inherent to weakly $1$-convex sets of those spaces.

Our final result refutes the validity of Lemma~\ref{lemm3} (b) and Lemma~\ref{lemm4} (b) for the spaces $\mathbb{R}^n$, $n\geqslant3$, as we construct examples of simultaneously weakly $1$-convex and weakly $1$-semiconvex open sets in $\mathbb{R}^n$, $n\geqslant3$, which non-empty $1$-nonconvexity-point sets are non-convex, bounded (or unbounded) and coincide with their $1$-nonsemiconvexity-point sets.

We give here briefly some directions for further research. Probably the most natural next question to study would be to investigate the general topological properties of $E_m^{\diamondsuit}$, $E\in\mathbf{WS^n_m}\setminus \mathbf{S^n_m}$,  and $G_m^{\triangle}$, $G\in\mathbf{WC^n_m}\setminus \mathbf{C^n_m}$, for $n\geqslant2$, $m\geqslant1$.  We expect $E_m^{\diamondsuit}$ and $G_m^{\triangle}$ to be open (closed) if $E$ and $G$ are open (closed). In addition, the question of estimating the number of components of the sets belonging to the class $\mathbf{WS^n_{n-1}}\setminus \mathbf{S^n_{n-1}}$ remains open for $n>2$.

\section{Main results}

Given two points $x,y\in\mathbb{R}^n$, we will denote by  $xy$ the open line segment between those points and by $\|x-y\|$ its length. Let also
$$U(y,\varepsilon):=\{x\in\mathbb{R}^n:\|x-y\|<\varepsilon\},\quad y\in\mathbb{R}^n,\quad \varepsilon>0;$$
$$S^{n-1}:=\{z\in \mathbb{R}^n: \|z\|=1\};$$
$$\eta_\alpha(y):=\{t\alpha+y: t\in[0,+\infty)\},\quad \alpha\in S^{n-1},\quad y\in\mathbb{R}^n;$$
$$
\gamma_{\alpha}(y):=\{t\alpha+y: t\in(-\infty,+\infty)\},\quad
\alpha\in S^{n-1},\quad y\in\mathbb{R}^n.
$$
\begin{lemma}\label{lemm1}
Suppose that a subset $E\subset \mathbb{R}^n$ is open and $E^{\diamondsuit}\ne\varnothing$ ($E^{\triangle}\ne\varnothing$). Let $y\in E^{\diamondsuit}$ ($y\in E^{\triangle}$). Then  for any ray $\eta_\alpha(y)$ (for any straight line $\gamma_\alpha(y)$), $\alpha\in S^{n-1}$, there exist points $x_\alpha(y)\in \eta_\alpha(y)\cap E$ ($x_\alpha(y)\in \gamma_\alpha(y)\cap E$) such that $U(x_\alpha(y),d(y))\subset E$, where $d(y)>0$ depends on only $y$ and does not depend on $\alpha$.
\end{lemma}

\begin{proof}[Proof.]

Let $O\in \mathbb{R}^n$ be the origin.
Consider the homeomorphism $\phi: \mathbb{R}^n\setminus \{O\}\rightarrow S^{n-1}\times (0,+\infty)$ defined by the formula
$$
\phi(z):=\left(\dfrac{z}{\|z\|},\|z\|\right).
$$
Let also $\sigma: S^{n-1}\times (0,+\infty)\rightarrow S^{n-1}$ be the central projection on the sphere, i.e.,
$$
\sigma(z):=\dfrac{z}{\|z\|}.
$$
Then $\sigma$ is open.

Fix an arbitrary point $y\in E^{\diamondsuit}$ ($y\in E^{\triangle}$). Without loss of generality, suppose that  $y=O$.

Let $z_\alpha$ be an arbitrary fixed point of $\eta_\alpha(y)\cap E$ (of $\gamma_{\alpha}(y)\cap E$),  $\alpha\in S^{n-1}$. Since $E$ is open, there exist open balls $U_\alpha:=U\left(z_\alpha,\varepsilon_\alpha\right)$,  $\alpha\in S^{n-1}$, such that  $\overline{U_\alpha}\subset E$.  Then the images $\sigma(U_\alpha)$, $\alpha\in S^{n-1}$, are open subsets of $S^{n-1}$ (open subsets of the projective space $\mathbb{R}\mathbf{P}^{n-1}$).
Moreover,
$$
\bigcup\limits_{\alpha\in S^{n-1}} \sigma(U_\alpha)
$$
is a cover of the unit sphere $S^{n-1}$ (of the projective space $\mathbb{R}\mathbf{P}^{n-1}$).
By the Heine-Borel theorem, there exists a subcover
$$
\bigcup\limits_{j\in M}
\sigma(U_{\alpha_j}),\quad \alpha_j\in S^{n-1},\quad j\in M,\quad M\,\, \mbox{is finite},
$$
of $S^{n-1}$ (of $\mathbb{R}\mathbf{P}^{n-1}$).

\begin{figure}[h]
	\centering
   \includegraphics[width=6 cm]{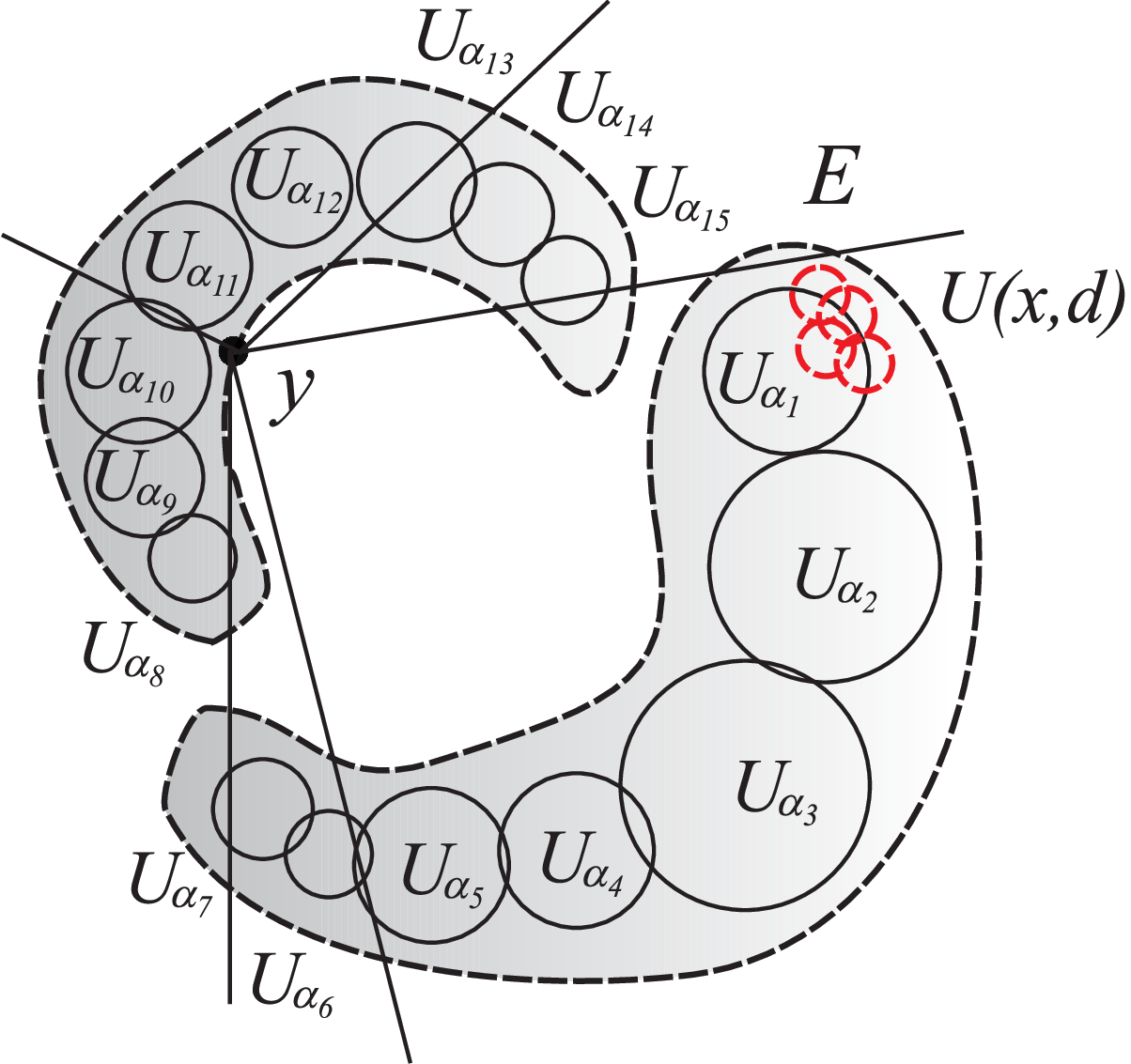}
	\caption{}\label{Fig5}
\end{figure}

 Let $E_i$, $i\in N\subseteq \mathbb{N}$, be the components of $E$. Then for any $j\in M$ there exists $i(j)\in N$ such that $\overline{U_{\alpha_j}}\subset E_{i(j)}$.

 Consider the distance functions
 $$d_i(x):=\inf\limits_{x^0\in \partial E_i}\|x-x^0\|,\quad x\in E_i,\quad i\in N.$$
  They are continuous  in the domains $E_i$, $i\in N$. Then their restrictions to the compacts $\overline{U_{\alpha_j}}$ attain their minimum values $d_{j}>0$ on that compacts, i.e.,
 $$
 d_{j}:=\min\limits_{x\in \overline{U_{\alpha_j}}}d_{i(j)}(x),\quad j\in M.
 $$
 Since $M$ is finite, there exists
 $$
 d:=\min\limits_{j\in M}d_j>0.
 $$
Then $U\left(x,d\right)\subset E$ for any point $x\in \overline{U_{\alpha_j}}$, $j\in M$; see Figure \ref{Fig5}. And for any $\alpha\in S^{n-1}$, there exists $j\in M$ such that  $\eta_\alpha(y)\cap  \overline{U_{\alpha_j}}\ne\varnothing$ ($\gamma_\alpha(y)\cap  \overline{U_{\alpha_j}}\ne\varnothing$) by the construction.
  \end{proof}

  \begin{theorem}\label{theor6}
Suppose that an open subset $E\subset \mathbb{R}^n$ belongs to the class $\mathbf{WS^n_1}\setminus \mathbf{S^n_1}$. Then $E^{\diamondsuit}$ is open.
  \end{theorem}
\begin{proof} Fix an arbitrary point $y\in E^{\diamondsuit}$ and show that it is an inner point of $E^{\diamondsuit}$.

Since $E$ is weakly $1$-semiconvex, it follows that $y\not\in \partial E$. Then there exists a number $\varepsilon_1>0$ such that $U(y,\varepsilon_1)\subset (\mathbb{R}^n\setminus \overline{E}$).

 By Lemma~\ref{lemm1}, for the fixed $y$ there exist points $x_\alpha\in \eta_\alpha(y)\cap E$, $\alpha\in S^{n-1}$, and a constant $d>0$ such that $U(x_\alpha,d)\subset E$.

 Let $\varepsilon:=\min\{\varepsilon_1,d\}$. Consider the neighborhood $U(y, \varepsilon)$ of the point $y$. Let $z\in U(y, \varepsilon)$ and let $\eta_\alpha(z)$, $\alpha\in S^{n-1}$, be an arbitrary ray with initial point at $z$. Draw the ray $\eta_\alpha(y)$ parallel to the ray $\eta_\alpha(z)$. Since $U(x_\alpha, \varepsilon)\subseteq U(x_\alpha, d)\subset E$ for the point $x_\alpha$ correspondent to $\eta_\alpha(y)$, it follows that $\eta_\alpha(z)\cap U(x_\alpha, \varepsilon)\ne\varnothing$ and, therefore, $\eta_\alpha(z)\cap E\ne\varnothing$ for any $\alpha\in S^{n-1}$; see Figure \ref{Fig3} a). Thus, $z$ is a $1$-nonsemiconvexity point of $E$. Since $z$ is arbitrary, it implies that all points of $U(y, \varepsilon)$ are $1$-nonsemiconvexity points of $E$. Hence, $y$ is an inner point of $E^{\diamondsuit}$.
 \end{proof}
\begin{figure}[h]
	\centering
   \includegraphics[width=11 cm]{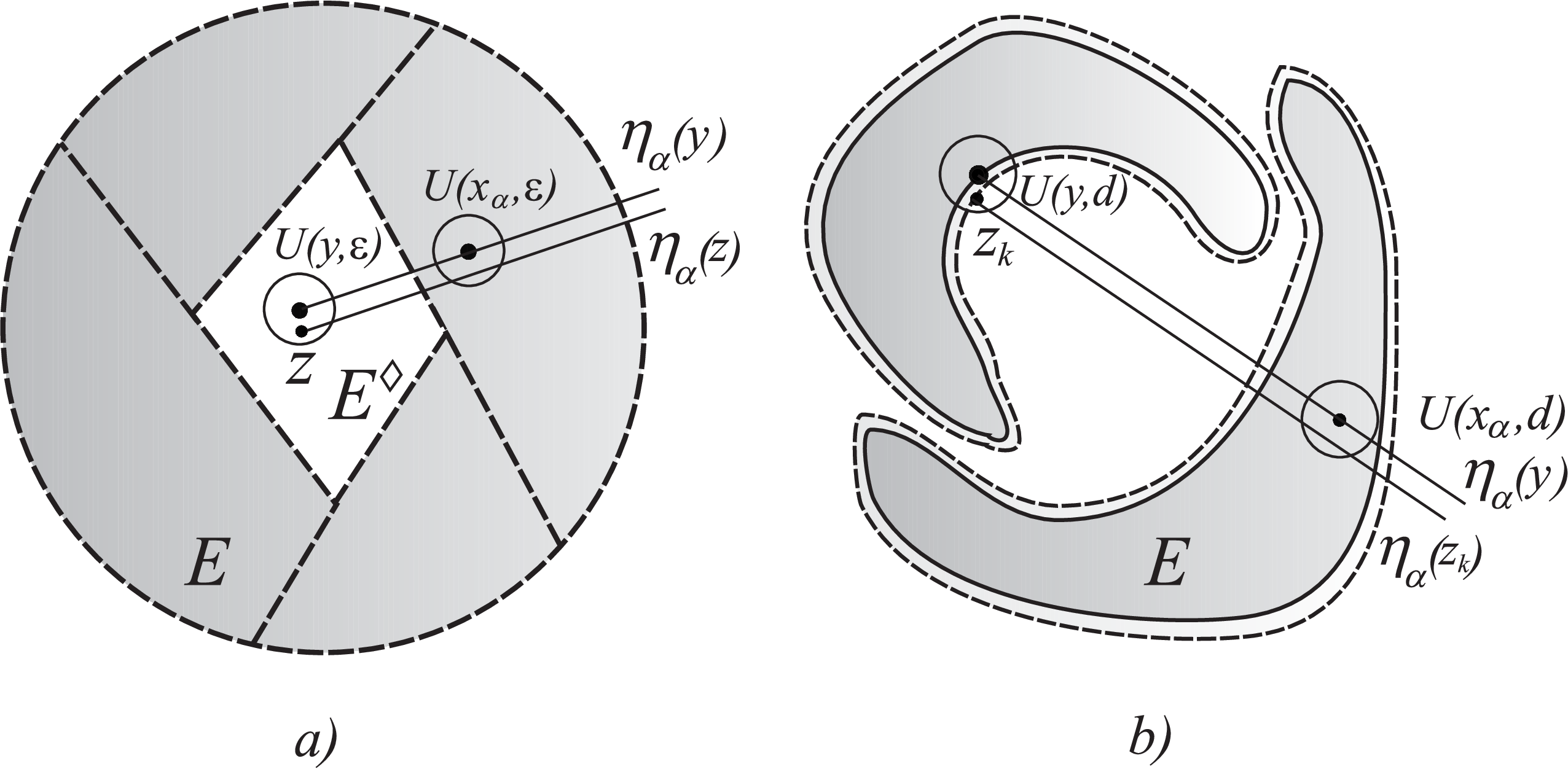}
	\caption{}\label{Fig3}
\end{figure}
 \begin{corollary} \label{corol1}
Suppose that an open subset $E\subset \mathbb{R}^n$ belongs to the class $\mathbf{WS^n_1}\setminus \mathbf{S^n_1}$. Then $E^{\diamondsuit}$ is weakly $1$-semiconvex.
\end{corollary}
 \begin{proof}[Proof.] Since  $E^{\diamondsuit}$ is open,  then for any point $y\in\partial E^{\diamondsuit}$, there exists a ray $\eta_{\alpha'}(y)$, $\alpha^{\prime}\in S^{n-1}$, not intersecting $E$. Then $\eta_{\alpha'}(y)\cap E^{\diamondsuit}=\varnothing$ by  Definition~\ref{def5}. Thus, $E^{\diamondsuit}$ is weakly $1$-semiconvex.
\end{proof}


\begin{theorem}\label{theor6_2}
Suppose that an open subset $E\subset \mathbb{R}^n$ belongs to the class $\mathbf{WC^n_1}\setminus \mathbf{C^n_1}$. Then $E^{\triangle}$ is open.
  \end{theorem}
  \begin{proof}[Proof.]
The scheme of proving this theorem is exactly the same as for Theorem~\ref{theor6}. We fix an arbitrary point of $E^{\triangle}$ and show that there exists a neighborhood of this point which belongs to $E^{\triangle}$. To do so, we use Lemma~\ref{lemm1} with respect to the points $y\in E^{\triangle}$ and the straight lines $\gamma_\alpha(y)$, $\alpha\in S^{n-1}$, and  we also replace the rays with the straight lines everywhere in the proof of Theorem~\ref{theor6}.
\end{proof}
\begin{corollary} \label{corol1_1}
Suppose that an open subset $E\subset \mathbb{R}^n$ belongs to the class $\mathbf{WC^n_1}\setminus \mathbf{C^n_1}$. Then $E^{\triangle}$ is weakly $1$-convex.
\end{corollary}

 \begin{theorem}\label{theor7}
Let $E\subset \mathbb{R}^n$ be a closed subset such that $\mathrm{Int}\, E\ne\varnothing$. If $E$ is weakly $1$-semiconvex, then $\mathrm{Int}\, E$ is weakly $1$-semiconvex.
\end{theorem}
\begin{proof}
Suppose that  $\mathrm{Int}\, E$ is not weakly $1$-semiconvex. Then there exists a $1$-nonsemiconvexity point $y\in\partial E$ of the set $\mathrm{Int}\, E$.

By Lemma~\ref{lemm1}, for the point $y$, there exist points $x_\alpha\in \eta_\alpha(y)\cap \mathrm{Int}\,E$, $\alpha\in S^{n-1}$, and a constant $d>0$ such that $U(x_\alpha,d)\subset \mathrm{Int}\,E$.

Consider the neighborhood $U(y, d)$ of the point $y$; see Figure \ref{Fig3}~b).
Since $E$ is weakly $1$-semiconvex, there exists a family of open weakly $1$-semiconvex sets $G_k$, $k=1,2,\ldots$, approximating $E$ from the outside. Then there exists an index $k_0$ such that $\partial G_k\cap U(y, d)\ne\varnothing$ for all $k\geqslant k_0$.  For each $k\geqslant k_0$, choose a point $z_k\in\partial G_k\cap U(y, d)$ and draw an arbitrary ray $\eta_\alpha(z_k)$, $\alpha\in S^{n-1}$,  with initial point at $z_k$. Consider the ray $\eta_\alpha(y)$ parallel to $\eta_\alpha(z_k)$.  Since $U(x_\alpha,d)\subset \mathrm{Int}\,E\subset E$ for the point $x_\alpha$ correspondent to $\eta_\alpha(y)$, it follows that $\eta_\alpha(z_k)\cap U(x_\alpha,d)\ne\varnothing$ and, therefore, $\eta_\alpha(z_k)\cap E\ne\varnothing$. Since $G_k\supset E$, $k=1,2,\ldots$, then $\eta_\alpha(z_k)\cap G_k\ne\varnothing$, $k\geqslant k_0$.

 Since the ray $\eta_\alpha(z_k)$ is arbitrary, the point $z_k\in\partial G_k$ is a $1$-nonsemiconvexity  point of $G_k$ for all $k\geqslant k_0$, which gives a contradiction.
 \end{proof}

 \begin{theorem}\label{theor7_1}
Let $E\subset \mathbb{R}^n$ be a closed subset such that $\mathrm{Int}\, E\ne\varnothing$. If $E$ is weakly $1$-convex, then $\mathrm{Int}\, E$ is weakly $1$-convex.
\end{theorem}
\begin{proof}
The proof of this theorem is the same as the proof of Theorem~\ref{theor7}. We only consider weakly $1$-convex sets instead of  weakly $1$-semiconvex and replace the rays with the straight lines everywhere in the proof of Theorem~\ref{theor7}.
\end{proof}

 \begin{proposition}\label{theor8}
Suppose that an open subset $E\subset \mathbb{R}^n$ belongs to the class $\mathbf{WS^n_1}\setminus \mathbf{S^n_1}$. Then there exists a collection of rays $\left\{\eta(x)\right\}_{x\in \partial E^{\diamondsuit}}$ such that
 \begin{enumerate}
     \item[$\bullet$] the set $\bigcup\limits_{x\in \partial E^{\diamondsuit}} \eta(x)\bigcup E^{\diamondsuit}$ does not contain rays emanating from $E^{\diamondsuit}$,
\item[$\bullet$] $\bigcup\limits_{x\in \partial E^{\diamondsuit}} \eta(x)\bigcap E=\varnothing.$
 \end{enumerate}
  \end{proposition}
  \begin{proof}
  Since $E^{\diamondsuit}$ is open, for any point $x\in \partial E^{\diamondsuit}$, there exists a ray $\eta(x)$ such that $\eta(x)\cap E=\varnothing$. Moreover, $\bigcup\limits_{x\in \partial E^{\diamondsuit}} \eta(x)\cup E^{\diamondsuit}$ does not contain any ray emanating from $E^{\diamondsuit}$, otherwise, a ray $\eta(y)\subset \bigcup\limits_{x\in \partial E^{\diamondsuit}} \eta(x)\cup E^{\diamondsuit}$, $y\in  E^{\diamondsuit}$, does not intersect $E$, which contradicts the definition of $1$-nonsemiconvexity point.
  \end{proof}
  \begin{proposition}\label{theor9}
Suppose that an open subset $E\subset \mathbb{R}^n$ belongs to the class $\mathbf{WC^n_1}\setminus \mathbf{C^n_1}$. Then there exists a collection of straight lines $\left\{\gamma(x)\right\}_{x\in \partial E^{\triangle}}$ such that
 \begin{enumerate}
     \item[$\bullet$] the set $\bigcup\limits_{x\in \partial E^{\triangle}} \gamma(x)\bigcup E^{\triangle}$ does not contain straight lines passing through $E^{\triangle}$,
\item[$\bullet$] $\bigcup\limits_{x\in \partial E^{\triangle}} \gamma(x)\bigcap E=\varnothing.$
 \end{enumerate}
  \end{proposition}
  \begin{proof}
  The statements are similar to the proof of Proposition~\ref{theor8}. We only consider straight lines instead of rays.
  \end{proof}

\begin{lemma}\label{lemm2}
There exists an open set $E^3\in\left(\mathbf{WS^3_1}\setminus \mathbf{S^3_1}\right)\cap \left(\mathbf{WC^3_1}\setminus\mathbf{C^3_1}\right)$  such that  the set $(E^3)^{\diamondsuit}=(E^3)^{\triangle}$ is bounded, connected, and non-convex.
\end{lemma}
\begin{proof}
  Let
 $$
  E^2:=\left(D\setminus P\right)\setminus \bigcup\limits_k \gamma^{k}\subset \mathbb{R}^2,
  $$
  where $D\subset \mathbb{R}^2$ is an open bounded convex subset, $P\subset \mathbb{R}^2$ is an open convex polygon such that $\overline{P}\subset D$, $\left\{\gamma^{k}\right\}_{k\in M}$, $M\subseteq\mathbb{N}$, is the finite collection of lines passing through the polygon sides. Then $E^2\in\left(\mathbf{WS^2_1}\setminus \mathbf{S^2_1}\right)\cap \left(\mathbf{WC^2_1}\setminus\mathbf{C^2_1}\right)$ and $(E^2)^{\diamondsuit}=(E^2)^{\triangle}=P$.

Consider the line segment $Oa_+^3$,
where $a_+^3\in \mathbb{R}^{3}$ is such that the angle between the vector $a_+^3$ and the unit vector $u$ of the axis $Ox_3$ belongs to the interval $(0,\pi/2)$.

 Let $E_+^3$ be a bounded oblique cylinder with the set $E^2$ at the base and elements parallel to $Oa_+^3$, i.e.,
 $$
E_+^3:=\{z\in \mathbb{R}^{3}: z=x+h,\,\, x\in E^2,\,\, h\in \{O\}\cup Oa_+^3\}.
$$

Let also  $E_-^3$  be the oblique cylinder symmetric to  $E_+^3$ with respect to the coordinate plane $x_1Ox_2$; see Figure~\ref{Fig16} a).

Let $\rho>0$ be the height of $E_+^3$, and $D_+$ be the orthogonal projection of the set $\{z\in \mathbb{R}^{3}: z=x+a_+^3,\,\, x\in D\}$ onto $x_1Ox_2$. Consider the following cylinders:
$$
D_+^3:=D_+\times \left(\rho,1\frac{1}{2}\rho\right)
$$
and $D_-^3$ that is the cylinder symmetric to  $D_+^3$ with respect to the coordinate plane $x_1Ox_2$.

Let
$$
E^3:=D_-^3\cup E_-^3\cup E_+^3\cup D_+^3,
$$
 see Figure~\ref{Fig16} a).
 \begin{figure}[h]
    \centering
   \includegraphics[width=12.5 cm]{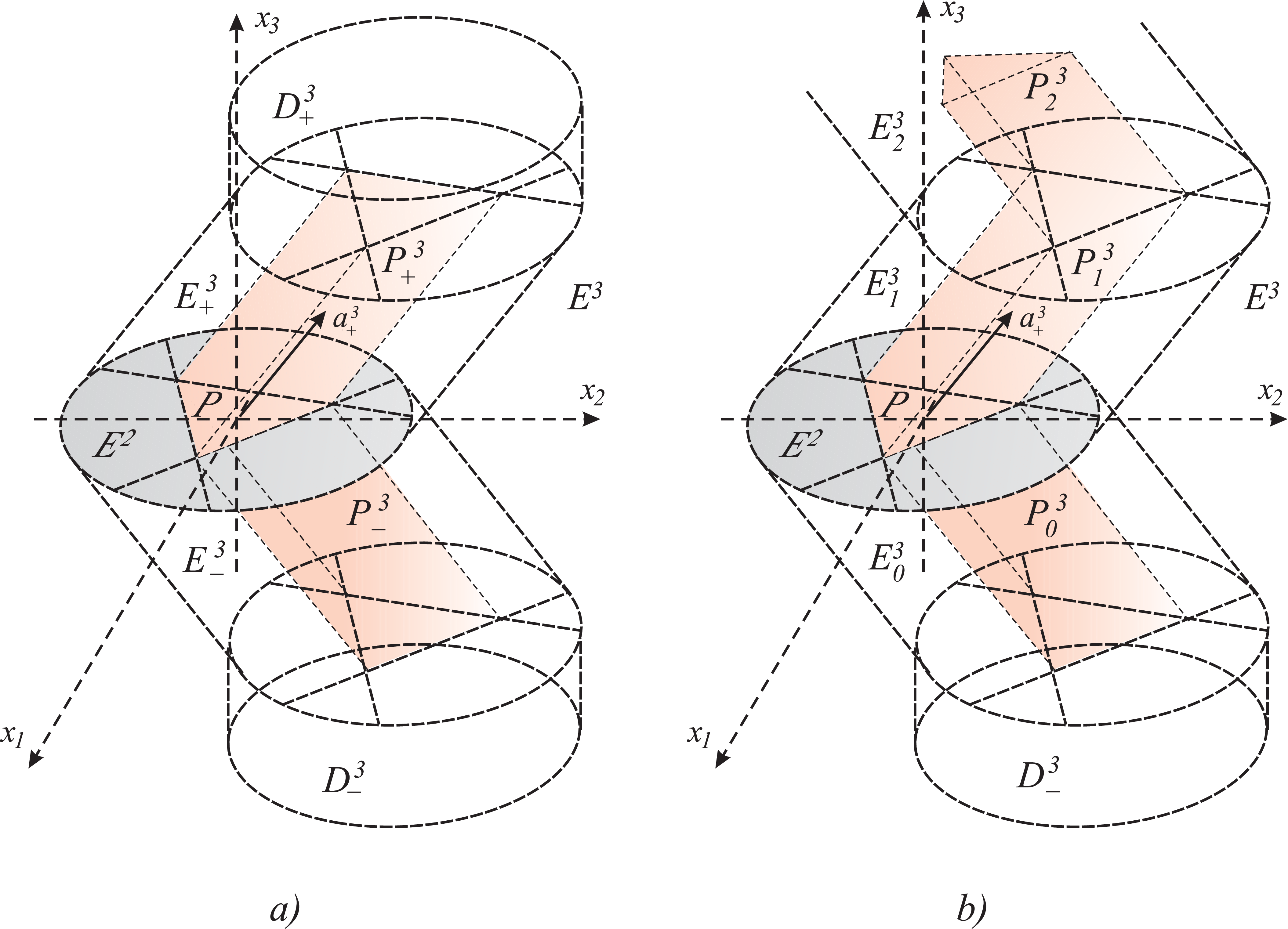}
    \caption{}\label{Fig16}
\end{figure}

Consider also the following polygonal oblique prisms:
$$
P_+^3:=\{z\in \mathbb{R}^{3}: z=x+h,\,\, x\in P,\,\, h\in \{O\}\cup Oa_+^3\},
$$
$$
P_-^3:=\{z\in \mathbb{R}^{3}: z=x+h,\,\, x\in P,\,\, h\in \{O\}\cup Oa_-^3\},
$$
where $a_-^3$ is the vector symmetric to  $a_+^3$ with respect to the coordinate plane $x_1Ox_2$.
Prove that $$(E^3)^{\triangle}=(E^3)^{\diamondsuit}= P_-^3\cup P_+^3.$$

First, show that $$(E^3)^{\triangle}\supset(E^3)^{\diamondsuit}\supset P_-^3\cup P_+^3.$$

Consider an arbitrary point $x\in P_-^3\cup P_+^3$. Then $x\in P_q^3$, $q\in\{-,+\}$. Let  $\eta(x)$ be an arbitrary ray emanating from  $x$. Show that $\eta(x)\cap E^3\ne\varnothing$.
\begin{enumerate}
  \item[1.]
If $\eta(x)$ intersects a lateral face of $P_q^3$, then consider  the projection, parallel to $Oa_q^3$, of $\eta(x)$ onto the coordinate plane $x_1Ox_2$. It  is a ray that we define by $\eta(x_0)$. The ray $\eta(x_0)$ emanates from the point $x_0\in x_1Ox_2$ which is the projection of $x$ onto $P\subset x_1Ox_2$.  Since $E^2$ is a flat weakly $1$-semiconvex set,  it implies that $\eta(x_0)\cap E^2\ne\varnothing$,  which gives that $\eta(x)\cap E^3_q\ne\varnothing$, therefore, $\eta(x)\cap E^3\ne\varnothing$.

\item[2.] If $\eta(x)$ intersects a base of $P_q^3$, then it intersects either $D_-^3\cup D_+^3$, which immediately gives that $\eta(x)\cap E^3\ne\varnothing$, or it intersects a lateral face of the other prism $P_{q'}^3$, ${q'}\in\{-,+\}$, ${q'}\ne q$,  by the construction, and further considerations are the same as in item 1, but for $P_{q'}^3$, a  point $x'\in \eta(x)\cap P_{q'}^3$, and the ray $\eta(x')\subset \eta(x)$. Then $\eta(x')\cap E^3\ne\varnothing$, therefor, $\eta(x)\cap E^3\ne\varnothing$.

\end{enumerate}
Moreover, if $x\in (E^3)^{\diamondsuit}$, then $x\in(E^3)^{\triangle}$.

Now, prove that $E^3\in\mathbf{WC^3_1}$ and, thus, $E^3\in\mathbf{WS^3_1}$, and
$$(E^3)^{\triangle}\subset(E^3)^{\diamondsuit}\subset P_-^3\cup P_+^3.$$

It is enough to show that if $z\not\in E^3\cup P_-^3\cup P_+^3,$ then $z\not\in (E^3)^{\triangle}$.  Let $L$ be the plane passing through $z$ parallel to the coordinate plane $x_1Ox_2$. Then the intersection $L\cap E^3$  is either 1) empty or 2) congruent to $D$, or 3) congruent to $E^2$.

1) Any straight line passing through $z$ in $L$ does not intersect $E^3$.

2) Since $L\cap E^3$ is convex in $L$, there exists a straight line passing through $z$ in $L$ and not intersecting $L\cap E^3$, therefore, not intersecting $E^3$.

3) $L\cap E^3\in \mathbf{WC^{2}_1}\setminus\mathbf{C^{2}_1}$ and $L\cap\left(P_-^3\cup P_+^3\right)=(L\cap E^3)^\triangle$ with respect to $L$. Since $z\not\in (L\cap E^3)^\triangle$, there exists a straight line passing through $z$  in $L$ and not intersecting $L\cap E^3$, therefore, not intersecting $E^3$.

The set $(E^3)^{\triangle}$ is bounded, connected, and non-convex, obviously.

\end{proof}

\begin{lemma}\label{lemm7}
There exists an open set $E^3\in\left(\mathbf{WS^3_1}\setminus \mathbf{S^3_1}\right)\cap \left(\mathbf{WC^3_1}\setminus\mathbf{C^3_1}\right)$  such that  the set $(E^3)^{\diamondsuit}=(E^3)^{\triangle}$ is unbounded, connected, and non-convex.
\end{lemma}
\begin{proof}
Consider the oblique cylinders

 \begin{align}\label{equ3}
E_{2k}^3&:=\{ z\in \mathbb{R}^{3}: z=x+2k\rho u,\,\, x\in E_-^3\},&\\
E_{2k+1}^3&:=\{z\in \mathbb{R}^{3}: z=x+2k\rho u,\,\, x\in E_+^3\}, & k=0,1,2,\ldots,
\end{align}
where
$$
E_-^3:=\{z\in \mathbb{R}^{3}: z=x+h,\,\, x\in E^2,\,\, h\in  Oa_-^3\},
$$
$$
E_+^3:=\{z\in \mathbb{R}^{3}: z=x+h,\,\, x\in E^2,\,\, h\in  \overline{Oa_+^3}\}.
$$

Now make sure that the unbounded open set
$$
E^3:=D_-^3\cup \bigcup\limits_{k=0}^\infty E^3_{k}
$$
belongs to the class $\left(\mathbf{WS^3_1}\setminus \mathbf{S^3_1}\right)\cap \left(\mathbf{WC^3_1}\setminus\mathbf{C^3_1}\right)$; see Figure~\ref{Fig16} b).

Consider the following polygonal oblique prisms.
 \begin{align*}
P_{2k}^3&:=\{z\in \mathbb{R}^{3}: z=x+2k\rho u,\,\,  x\in P_-^3\},&\\
P_{2k+1}^3&:=\{z\in \mathbb{R}^{3}: z=x+2k\rho u,\,\,  x\in P_+^3\}, & k=0,1,2,\ldots,
\end{align*}
where
$$
P_-^3:=\{z\in \mathbb{R}^{3}: z=x+h,\,\, x\in P,\,\, h\in Oa_-^3\},
$$
$$
P_+^3:=\{z\in \mathbb{R}^{3}: z=x+h,\,\, x\in P,\,\, h\in \overline{Oa_+^3}\}.
$$
At this point, we choose $a_+^3$ such that the set $\bigcup\limits_{k=0}^\infty P_{k}^3$ does not contain any ray.

Prove that
$$(E^3)^{\triangle}=(E^3)^{\diamondsuit}= \bigcup\limits_{k=0}^\infty P_{k}^3.$$

First, show that $$(E^3)^{\triangle}\supset(E^3)^{\diamondsuit}\supset \bigcup\limits_{k=0}^\infty P_{k}^3.$$

Define the bottom base of each prism $P^3_{k}$,  $ k=0,1,2,\ldots$, by $P^-_k$. Consider an arbitrary point $x\in \bigcup\limits_{k=0}^\infty P_{k}^3$. Then $x\in P^3_q$, $q\in \{0,1,2,\ldots\}$. Let  $\eta(x)$ be an arbitrary ray emanating from  $x$.
\begin{enumerate}
  \item[1.]
If $\eta(x)$ intersects a lateral face of the prism $P^3_q$, then consider  the projection, parallel to the lateral edges of $P^3_q$, of $\eta(x)$ onto the plane $L\supset P^-_q$. It  is a ray that we define by $\eta(x_0)$. The ray $\eta(x_0)$ emanates from the point $x_0\in P^-_q$ which is the projection of $x$ onto $L$.  Since $L\cap E^3$ is a flat weakly $1$-semiconvex set as a set congruent to $E^2$,  it implies that $\eta(x_0)\cap (L\cap E^3)\ne\varnothing$,  which gives that $\eta(x)\cap E^3_q\ne\varnothing$, therefore, $\eta(x)\cap E^3\ne\varnothing$.

\item[2.] If $\eta(x)$ intersects a base of $P^3_q$, then it intersects either $D_-^3$, which immediately gives that $\eta(x)\cap E^3\ne\varnothing$, or it intersects a lateral face of another prism $P^3_{q'}$, $q'\in \{0,1,2,\ldots\}$, by the construction, and further considerations are the same as in item 1, but for a  point $x'\in \eta(x)\cap P^3_{q'}$ and the ray $\eta(x')\subset \eta(x)$. Then $\eta(x')\cap E^3\ne\varnothing$, therefor, $\eta(x)\cap E^3\ne\varnothing$.

\end{enumerate}
Moreover, if $x\in (E^3)^{\diamondsuit}$, then $x\in(E^3)^{\triangle}$.

The proof of the fact that $E^3\in\mathbf{WC^3_1}$ and, therefore, $E^3\in\mathbf{WS^3_1}$, and
$$(E^3)^{\triangle}\subset(E^3)^{\diamondsuit}\subset \bigcup\limits_{k=0}^\infty P^3_k$$
is the same as in the proof of Lemma~\ref{lemm2}.
\end{proof}

\begin{theorem}
There exists an open set $E^n\in\left(\mathbf{WS^n_1}\setminus \mathbf{S^n_1}\right)\cap \left(\mathbf{WC^n_1}\setminus\mathbf{C^n_1}\right)$, $n\geqslant 3$,  such that  the set $(E^n)^{\diamondsuit}=(E^n)^{\triangle}$ is bounded (or unbounded), connected, and non-convex.
\end{theorem}
\begin{proof} Prove theorem by the induction. For $n=3$, the theorem holds by Lemmas~\ref{lemm2} and~\ref{lemm7}. Suppose that, for $n>3$, an open set
   $E^{n-1}\in\left(\mathbf{WS^{n-1}_1}\setminus \mathbf{S^{n-1}_1}\right)\cap \left(\mathbf{WC^{n-1}_1}\setminus\mathbf{C^{n-1}_1}\right)$, and $P^{n-1}:=(E^{n-1})^{\diamondsuit}=(E^{n-1})^{\triangle}$ is bounded (or unbounded), connected, and non-convex.

Consider the following sets:
$$
\widetilde{E}^n:=E^{n-1}\times (-1,1),
$$
$$
D_-^n:=D^{n-1}\times \left(-1\frac{1}{2},-1\right),\quad D_+^n:=D^{n-1}\times \left(1,1\frac{1}{2}\right),
$$
where $D^{n-1}\subset \mathbb{R}^{n-1}$ is the convex hull of $E^{n-1}$,
$$
E^n:=D_-^n\cup \widetilde{E}^n\cup D_+^n.
$$

First, show that $$(E^n)^{\triangle}\supset(E^n)^{\diamondsuit}\supset P^{n-1}\times (-1,1).$$

Consider an arbitrary point $x\in P^{n-1}\times (-1,1)$, and an arbitrary ray $\eta(x)$ emanating from  $x$. If $\eta(x)\cap (D_-^n\cup D_+^n)\ne\varnothing$, then $\eta(x)\cap E^n\ne\varnothing$.  If $\eta(x)\cap (D_-^n\cup D_+^n)=\varnothing$, then consider  the orthogonal projection of $\eta(x)$ onto the coordinate subspace $\mathbb{R}^{n-1}$. It  is a ray $\eta(x_0)$ emanating from the point $x_0\in P^{n-1}$ which is the orthogonal projection of $x$ onto $\mathbb{R}^{n-1}$.  Therefore, $R:=\eta(x_0)\cap E^{n-1}\ne\varnothing$,  which gives that $\eta(x)\cap(R\times (-1,1))\ne\varnothing$. Then $\eta(x)\cap E^n\ne\varnothing$.

Moreover, if $x\in (E^n)^{\diamondsuit}$, then $x\in(E^n)^{\triangle}$.

Now, prove that $E^n\in\mathbf{WC^n_1}$, therefore, $E^n\in\mathbf{WS^n_1}$, and
$$(E^n)^{\triangle}\subset(E^n)^{\diamondsuit}\subset P^{n-1}\times (-1,1).$$

It is enough to show that if $$z\not\in E^n\cup \left(P^{n-1}\times (-1,1)\right),$$ then $z\not\in (E^n)^{\triangle}$.  Let $L$ be the $(n-1)$-dimensional plane passing through $z$ parallel to the coordinate subspace $\mathbb{R}^{n-1}$. Then the intersection $L\cap E^n$  is either 1) empty or 2) congruent to $D^{n-1}$, or 3) congruent to $E^{n-1}$.

1) Any straight line passing through $z$ in $L$ does not intersect $E^n$.

2) Since $L\cap E^n$ is convex in $L$, there exists a straight line passing through $z$ in $L$ and not intersecting $L\cap E^n$, therefore, not intersecting $E^n$.

3) $L\cap E^n\in \mathbf{WC^{n-1}_1}\setminus\mathbf{C^{n-1}_1}$ and $L\cap\left(P^{n-1}\times (-1,1)\right)=(L\cap E^n)^\triangle$ with respect to $L$. Since $z\not\in L\cap\left(P^{n-1}\times (-1,1)\right)$, there exists a straight line passing through $z$ in $L$ and not intersecting $L\cap E^n$, therefore, not intersecting $E^n$.

The set $P^{n-1}\times (-1,1)$ is bounded (or unbounded), connected, and non-convex, obviously.
\end{proof}
  \section*{Declarations}
 This work was supported by a grant from the Simons Foundation (1290607,T.M.O.). The author declare no potential conflict of interest with respect to the research, authorship and publication of this article. All necessary data are included into the paper.

\end{document}